\documentclass[11pt,reqno]{amsart}

\usepackage[T1]{fontenc}
\usepackage{lmodern}
\usepackage{microtype}
\usepackage{amsmath,amssymb,amsthm,mathtools,mathrsfs}
\usepackage[a4paper,margin=28mm,headheight=14pt]{geometry}
\usepackage{enumitem}
\usepackage[
  colorlinks=true,
  linkcolor=red,
  citecolor=blue,
  urlcolor=blue
]{hyperref}

\allowdisplaybreaks[2]
\numberwithin{equation}{section}
\setlist[enumerate]{label=\textup{(\roman*)},leftmargin=*,itemsep=3pt}

\newtheorem{theorem}{Theorem}[section]
\newtheorem{proposition}[theorem]{Proposition}
\newtheorem{lemma}[theorem]{Lemma}

\theoremstyle{definition}
\newtheorem{definition}[theorem]{Definition}
\theoremstyle{remark}
\newtheorem{remark}[theorem]{Remark}

\newcommand{\C}{\mathbb C}
\newcommand{\R}{\mathbb R}
\newcommand{\Q}{\mathbb Q}
\newcommand{\Z}{\mathbb Z}
\newcommand{\HH}{\mathbb H}

\newcommand{\Tg}{\mathcal T_g}
\newcommand{\Mg}{\mathcal M_g}

\newcommand{\NL}{\operatorname{NL}}
\newcommand{\SL}{\operatorname{SL}}
\newcommand{\PSL}{\operatorname{PSL}}

\newcommand{\tr}{\operatorname{tr}}
\newcommand{\Div}{\operatorname{div}}
\newcommand{\orb}{\mathrm{orb}}
\newcommand{\Aut}{\operatorname{Aut}}
\newcommand{\Hom}{\operatorname{Hom}}

\newcommand{\cV}{\mathbb V}
\newcommand{\cH}{\mathcal H}

\newcommand{\End}{\mathrm{End}}
\newcommand{\ad}{\operatorname{ad}}
\newcommand{\id}{\operatorname{id}}
\title[On the algebraicity of non-abelian NL loci of rank-two real local systems]{On the algebraicity of non-abelian Noether--Lefschetz loci of rank-two real local systems}
\author[Tianzhi Hu]{Tianzhi Hu}
\address{
School of Mathematics and Statistics, Wuhan University, Wuhan, Hubei 430072, P.R. China
}
\email{hutianzhi@whu.edu.cn}

\author[Kang Zuo]{Kang Zuo}
\address{
School of Mathematics and Statistics, Wuhan University, Wuhan, Hubei 430072, P.R. China;
Shanghai Institute for Mathematics and Interdisciplinary Sciences (SIMIS), Shanghai, P.R. China;
Institut f\"ur Mathematik, Johannes Gutenberg-Universit\"at Mainz, Mainz 55099, Germany
}
\email{zuok@uni-mainz.de}

\date{}

\begin{document}
\begin{abstract}
As a non-abelian analogue of the Hodge locus, Simpson introduced the non-abelian Noether--Lefschetz locus and conjectured its algebraicity for $\mathbb Z$PVHS. In this paper, we study this question on the moduli space of curves \(\mathcal M_g\). For a non-unitary representation
\[
\rho:\pi_1(\Sigma_g)\longrightarrow \mathrm{SL}_2(\mathbb R)
\]
which admits a $\mathbb R$PVHS of weight one, we prove that algebraicity of a positive-dimensional non-abelian Noether--Lefschetz component is equivalent to discreteness of $\operatorname{im}\rho$, and in this case the component is precisely a marked fixed-target orbifold Hurwitz component. As applications, we construct two explicit rational families by slit surgery: one with non-discrete monodromy and non-algebraic Noether--Lefschetz image, and another with discrete monodromy whose period map is nevertheless non-uniformizing. The latter gives an affirmative answer to a question of Baldi--Lam concerning $\mathbb Q$PVHS.
\end{abstract}

\maketitle

\section{Introduction}\label{sec:intro}

Let \(\Sigma_g\) be a closed oriented smooth surface of genus \(g\ge2\), and
let \(\Tg\) be its Teichm\"uller space. We write
\[
 \pi:\mathcal C_g\longrightarrow\Tg
\]
for the universal marked curve and
\[
 p:\Tg\longrightarrow\Mg
\]
for the map forgetting the marking. Thus a point
\(t:=[X,\phi]\in\Tg\) consists of a compact Riemann surface \(X\) together
with a marking \(\phi:\Sigma_g\to X\).

Fix a representation
\[
 \rho:\pi_1(\Sigma_g)\longrightarrow\SL_2(\R).
\]
For a marked curve \(t=[X,\phi]\in\Tg\), let
\(\cV_{\rho,X}\) denote the rank-two real local system on \(X\)
induced by \(\rho\) via the marking, and set
\[
 \cV_{\rho,X,\C}:=\cV_{\rho,X}\otimes_{\R}\C.
\]
Via the markings, \(\rho\) determines an isomonodromic family of flat
bundles on the fibers of \(\pi\), hence a de Rham section
\[
 \sigma_{\mathrm{dR}}:\Tg\longrightarrow
 M_{\mathrm{dR}}(\mathcal C_g/\Tg).
\]
Fiberwise non-abelian Hodge theory, originating in the work of
Uhlenbeck--Yau, Hitchin, and Simpson
\cite{UY86,Hitchin87,Simpson92}, gives a relative real-analytic
correspondence \cite{CTW,HSYZ}
\[
 \mathrm{NHC}:M_{\mathrm{dR}}(\mathcal C_g/\Tg)
 \longrightarrow M_{\mathrm{Dol}}(\mathcal C_g/\Tg),
\]
and therefore a real-analytic Dolbeault section
\[
 \sigma_{\mathrm{Dol}}
 :=\mathrm{NHC}\circ\sigma_{\mathrm{dR}}.
\]

\medskip

Motivated by the classical Noether--Lefschetz locus and the theory of Hodge
loci, Simpson introduced the \textbf{non-abelian Noether--Lefschetz locus}
\cite{Simpson97}. For the representation \(\rho\), let
\begin{equation}\label{eq:NLC}
 \NL_{\C}(\rho)
 :=
 \left\{
 [X,\phi]\in\Tg:
 \cV_{\rho,X,\C}\text{ underlies a complex PVHS}
 \right\}.
\end{equation}
Simpson proved that this locus is a complex-analytic subvariety
\cite[Theorem~12.1]{Simpson97}. As a non-abelian analogue of the
Cattani--Deligne--Kaplan algebraicity theorem for Hodge loci
\cite{CDK}, he conjectured that, for integral representations over a
quasi-projective base, the corresponding non-abelian
Noether--Lefschetz locus is algebraic; see
\cite[Conjecture~12.3]{Simpson97}. Recently, Engel and Tayou proved this algebraicity conjecture for local systems with $\mathbb Q$-anisotropic monodromy; see \cite{EngelTayou26}.

\medskip

For a representation \(\rho:\pi_1(\Sigma_g)\to\SL_2(\R)\), we consider the weight-one real locus
\begin{equation}\label{eq:NLR}
 \NL_{\R}(\rho):=
 \left\{
 [X,\phi]\in\Tg:
 \cV_{\rho,X}\text{ underlies a real PVHS of weight one}
 \right\}.
\end{equation}
We remark that in the real rank-two setting, the two loci above coincide:
\[
\NL_{\C}(\rho)=\NL_{\R}(\rho).
\]

The natural algebraicity question in the universal-curve setting is whether a positive-dimensional component of the non-abelian Noether-Lefschetz locus projects to an algebraic subvariety of \(\Mg\). We prove the following real rank-two theorem.

\begin{theorem}\label{thm:main}
Let \(g\ge2\), and let
\[
 \rho:\pi_1(\Sigma_g)\longrightarrow\SL_2(\R)
\]
be non-unitary and suppose that \(\rho\) admits a weight-one $\mathbb R$PVHS for some marked complex structure. Let \(U\) be
a positive-dimensional connected component of \(\NL_{\R}(\rho)\). The
following are equivalent:
\begin{enumerate}
\item the image \(p(U)\subset\Mg\) is a closed algebraic subvariety;
\item the image \(\rho\bigl(\pi_1(\Sigma_g)\bigr)\subset\SL_2(\R)\) is discrete;
\item \(U\) is a marked fixed-target orbifold Hurwitz component in the
sense of Definition~\ref{def:fixed-target-hurwitz}.
\end{enumerate}
\end{theorem}

\medskip

As applications, we use the slit construction to produce two explicit non-unitary representations $\rho:\pi_1(\Sigma_g)\to\SL_2(\Q)$ which admit weight-one $\mathbb Q$PVHS. For the first, the monodromy image is non-discrete, yielding a non-algebraic Noether--Lefschetz image; for the second, the monodromy image is discrete, yielding an algebraic Noether--Lefschetz image.

For the rational examples below, \(\NL_{\Q}(\rho)\) denotes the locus where
the induced \(\Q\)-local system underlies a weight-one $\mathbb Q$PVHS. In rank two this agrees with \(\NL_{\R}(\rho)\). Indeed, since
\(\rho\) takes values in \(\SL_2(\Q)\), the standard symplectic form on
\(\Q^2\) is \(\rho\)-invariant. Any real polarization is a nonzero real
multiple of this form, and hence, after changing the sign if necessary,
the standard rational symplectic form polarizes the same weight-one Hodge
structure.

\begin{proposition}\label{prop:non-discrete-example}
For every \(g\ge3\), there exists an absolutely irreducible, non-unitary
representation
\[
 \rho_{\mathrm{nd}}:\pi_1(\Sigma_g)\longrightarrow\SL_2(\Q)
\]
such that:
\begin{enumerate}
\item \(\NL_{\Q}(\rho_{\mathrm{nd}})\subset\Tg\) is a smooth complex surface;
\item \(\rho_{\mathrm{nd}}\) preserves no full integral lattice in
\(\Q^2\);
\item the image of \(\rho_{\mathrm{nd}}\) in \(\SL_2(\R)\) is not discrete.
\end{enumerate}
\end{proposition}
\begin{remark}\label{rem:Baldi-Lam}
By Theorem~\ref{thm:main}, for every connected component
\(U\subset\NL_{\Q}(\rho_{\mathrm{nd}})\), its image $p(U)\subset\Mg$ is not a closed algebraic subvariety. Our constructions are closely related to the recent work of Baldi and Lam
\cite{BaldiLam26} on non-integral \(\Q\)VHS.

\begin{enumerate}
\item
Baldi and Lam construct non-integral \(\Q\)VHS using explicit
Fenchel--Nielsen-type parametrizations of Teichm\"uller components of
character varieties, together with arithmetic control of trace rings.
Proposition~\ref{prop:non-discrete-example} gives a different explicit
source of non-integral \(\Q\)PVHS, obtained from a rational Fuchsian group
by slit surgery and the associated branched hyperbolic developing map.

\item
Baldi and Lam show that the Cattani--Deligne--Kaplan algebraicity theorem does not extend to non-integral \(\Q\)VHS.
Proposition~\ref{prop:non-discrete-example}, together with
Theorem~\ref{thm:main}, gives the corresponding non-abelian phenomenon:
once the integrality hypothesis is removed, a positive-dimensional
non-abelian Noether--Lefschetz component need not have algebraic image in
\(\Mg\), in contrast with Simpson's algebraicity conjecture.
\end{enumerate}
\end{remark}

The complementary construction has discrete monodromy.  

\begin{proposition}\label{prop:discrete-example}
For every \(g\ge4\), there exists an absolutely irreducible, non-unitary
representation
\[
 \rho_{\mathrm{disc}}:\pi_1(\Sigma_g)\longrightarrow\SL_2(\Q)
\]
such that:
\begin{enumerate}
\item \(\NL_{\Q}(\rho_{\mathrm{disc}})\subset\Tg\) is a smooth complex surface;
\item \(\rho_{\mathrm{disc}}\) preserves no full integral lattice in
\(\Q^2\);
\item the image of \(\rho_{\mathrm{disc}}\) in \(\SL_2(\R)\) is discrete and projects to a cocompact Fuchsian group.
\end{enumerate}
Moreover, the slit construction shows that
\(\rho_{\mathrm{disc}}\) admits a rational weight-one PVHS with
discrete monodromy whose period map is not uniformizing.
\end{proposition}
\begin{remark}\label{rem:discrete-example}
By Theorem~\ref{thm:main}, every connected component
\(U\subset\NL_{\Q}(\rho_{\mathrm{disc}})\) is a marked fixed-target
orbifold Hurwitz component, and \(p(U)\subset\Mg\) is a closed algebraic
surface.  Since the resulting \(\Q\)PVHS has discrete monodromy but is not
uniformizing, Proposition~\ref{prop:discrete-example} also gives an
affirmative answer to Baldi--Lam \cite[Question~7.5]{BaldiLam26}.
\end{remark}

The proof is organized as follows. Section~\ref{sec:local} studies the local geometry of the fixed-monodromy locus. Using the tangent-space formula together with the rank-two period map, we prove the smoothness and dimension statement and establish the germ-rigidity result in Lemma~\ref{lem:local}. Section~\ref{sec:extension} passes from the local analytic locus to an algebraic family: Proposition~\ref{prop:extension} provides the required algebraic lifting and extends the fiber representation to the total space. Section~\ref{sec:main-proof} proves the three equivalences. Its main input in the algebraic-to-discrete direction is the Corlette--Simpson--Loray--Pereira--Touzet factorization theorem, Theorem~\ref{thm:CS-LPT-alternative}, which reduces the argument to an orbicurve factorization; the remaining steps combine the Milnor--Wood inequality with a Hurwitz-space dimension comparison. Conversely, discreteness identifies \(U\) with a fixed-target orbifold Hurwitz component whose forgetful morphism to \(\Mg\) is finite. Section~\ref{sec:examples} concludes with two explicit slit constructions realizing, respectively, the non-discrete and discrete cases.

\section{Local geometry of the non-abelian Noether--Lefschetz locus}\label{sec:local}
Under the assumptions of Theorem~\ref{thm:main}, write
\[
 \bar\rho:\pi_1(\Sigma_g)\longrightarrow\PSL_2(\R)
\]
for the projectivization. Assume that $[X,\phi]\in\NL_{\R}(\rho)$. A real PVHS of rank
two and weight one is equivalently described by a
$\bar\rho$-equivariant holomorphic period map
\begin{equation*}
 F:\widetilde X\longrightarrow\HH.
\end{equation*}
The period map is nonconstant, since otherwise $\bar\rho$ would fix a
point of $\HH$ and $\rho$ would be unitary.

Let $L\subset \cV_{\rho,X,\C}$ be the Hodge line.  Since
$\det\cV_{\rho,X,\C}\simeq\mathcal O_X$, the associated graded Higgs
bundle is
\begin{equation*}
 \mathcal E:=L\oplus L^{-1},
 \qquad
 \theta:\mathcal E\longrightarrow\mathcal E\otimes K_X,
\end{equation*}
where, with respect to this decomposition,
\begin{equation}\label{eq:higgs-matrix}
 \theta=
 \begin{pmatrix}
 0&0\\
 \beta&0
 \end{pmatrix},
 \qquad
 \beta\in H^0(X,K_XL^{-2}).
\end{equation}

 Set $\mathcal N:=\Hom(L,L^{-1})\simeq L^{-2}.$ Thus $\mathcal N$ is the descent of $F^*T_{\HH}$, and the Higgs field
$\beta$ is equivalently the differential of the period map,
\[
 dF:T_X\longrightarrow\mathcal N.
\]
Put
\begin{equation}\label{eq:euler}
 e:=e(\bar\rho):=-\deg\mathcal N=2\deg L>0.
\end{equation}
The number $e$ is the signed Euler number of the oriented projective
representation.  Let
\begin{equation*}
 D:=\Div(dF)=\Div(\beta),
 \qquad
 b:=\deg D=2g-2-e.
\end{equation*}
Then
\begin{equation}\label{eq:normal-sequence}
 0\longrightarrow T_X\xrightarrow{\,dF\,}\mathcal N
 \longrightarrow\mathcal N|_D\longrightarrow0.
\end{equation}
In particular $b\ge0$; this follows directly from the existence of the
nonzero morphism $T_X\to\mathcal N$.

The image of \(\bar\rho\) in \(\PSL_2(\R)\) is nonelementary.  Indeed, an elementary subgroup
of $\PSL_2(\R)$ has Euler number zero: after passing to a subgroup of
index at most two, it fixes either a point of $\HH$ or a point of
$\partial\HH$, and the associated flat circle bundle admits a section.
Since $e>0$, this is impossible.  In particular, the image of \(\bar\rho\) in \(\PSL_2(\C)\) is
Zariski dense.

The local geometry of the non-abelian Noether--Lefschetz locus is
controlled by the non-abelian Kodaira--Spencer map introduced by \cite{FuSheng,HSZ}.  We recall the
deformation-theoretic description in the form needed below.  Let
\(t=[X,\phi]\in\NL_{\C}(\rho)\), and assume that the associated graded
Higgs bundle \((\mathcal E,\theta)\) is stable.  The Higgs field induces a
morphism of deformation complexes
\[
 (T_X,0)\longrightarrow
  \mathcal K^\bullet_{\theta}
 :=\left(
 \End(\mathcal E)\xrightarrow{\ad(\theta)}
 \End(\mathcal E)\otimes K_X
 \right),
\]
and hence a map
\begin{equation}\label{eq:non-abelian-KS-intro}
 \theta_*:
 H^1(X,T_X)\longrightarrow
 \mathbb H^1\!\left(
 X,\,
 \End(\mathcal E)\xrightarrow{\ad(\theta)}
 \End(\mathcal E)\otimes K_X
 \right).
\end{equation}
Hu--Sun--Yang--Zuo identify the Zariski tangent space of the
non-abelian Noether--Lefschetz locus with the kernel of this
non-abelian Kodaira--Spencer map \cite{HSYZ}.  Since the
Kodaira--Spencer map of the universal marked curve is the identity,
their formula gives
\begin{equation}\label{eq:HSYZ-tangent-intro}
 T_t^{\mathrm{Zar}}\NL_{\C}(\rho)=\ker\theta_* .
\end{equation}
This will be used to identify the tangent space of the non-abelian Noether--Lefschetz locus.

\medskip

For $\mathcal E=L\oplus L^{-1}$, we have $\End(\mathcal E)
 =\mathcal O_X \oplus \mathcal N^{-1}\oplus
 \mathcal N \oplus \mathcal O_X.$ Contraction with the Higgs field defines a morphism of complexes
\begin{equation}\label{eq:theta-morphism}
 (T_X\longrightarrow0)
 \longrightarrow
 \mathcal K^\bullet_{\theta},
 \qquad
 v\longmapsto\iota_v\theta
 =
 \begin{pmatrix}
 0&0\\
 dF(v)&0
 \end{pmatrix}.
\end{equation}

\begin{lemma}[Local structure and germ rigidity]\label{lem:local}
The reduced fixed-representation locus $\NL_{\R}(\rho)$ is, wherever
nonempty, a closed complex submanifold of $\Tg$ of dimension
$b$.  At $[X,\phi]$ its tangent space is given by the following equivalent
description:
\begin{equation}\label{eq:tangent-NL}
 T_{[X,\phi]}\NL_{\R}(\rho)
 =
 \ker\!\left(
 H^1(X,T_X)
 \xrightarrow{\,H^1(dF)\,}
 H^1(X,\mathcal N)
 \right).
\end{equation}
Moreover, among absolutely irreducible non-unitary representations with
the same signed Euler number, the germ of $\NL_{\R}$ at
$[X,\phi]$ determines the projective representation up to
$\PSL_2(\R)$-conjugacy.
\end{lemma}

\begin{proof}
By \cite[Lemma 6.2]{HSZ}, the non-abelian Kodaira--Spencer map $\theta_*$ factors
through the natural injection of $H^1(X,\ker\ad(\theta))$ into the
corresponding hypercohomology group. Notice that by \eqref{eq:higgs-matrix} we have
\begin{equation}\label{eq:centralizer-theta}
 \ker\!\left(
 \ad(\theta):\End(\mathcal E)\to\End(\mathcal E)\otimes K_X
 \right)
 =
 \mathcal O_X\!\cdot\!\id_{\mathcal E}\oplus\mathcal N.
\end{equation}
 In the present rank-two setting,
\eqref{eq:theta-morphism} and \eqref{eq:centralizer-theta} identify the
induced map $\theta_*$ with
\[
H^1(X,T_X)\xrightarrow{\,H^1(dF)\,}H^1(X,\mathcal N)
 \hookrightarrow H^1(X,\ker\ad(\theta)),\]
which together with \eqref{eq:HSYZ-tangent-intro} gives \eqref{eq:tangent-NL}.

\medskip

To prove smoothness, consider separately the fixed-holonomy deformation
complex of the equivariant period map,
\begin{equation*}
 \mathcal C_F^\bullet
 :=
 \left[
 T_X\xrightarrow{\,dF\,}\mathcal N
 \right],
\end{equation*}
with $T_X$ in degree zero.  By \eqref{eq:normal-sequence}, we have $\mathcal C_F^\bullet\simeq\mathcal N|_D[-1].$ Hence
\[
 \mathbb H^2(X,\mathcal C_F^\bullet)=0,
 \qquad
 \dim\mathbb H^1(X,\mathcal C_F^\bullet)
 =
 h^0(X,\mathcal N|_D)
 =
 \deg D=b.
\]
Thus fixed-holonomy deformations are unobstructed of dimension $b$.
Since $\deg\mathcal N=-e<0$, one has $H^0(X,\mathcal N)=0$, so
forgetting the developing map is locally immersive.

On the other hand, the long exact sequence associated with
\eqref{eq:normal-sequence} gives
\[
 \dim\ker\!\left(
 H^1(X,T_X)\xrightarrow{\,H^1(dF)\,}H^1(X,\mathcal N)
 \right)
 =
 h^0(X,\mathcal N|_D)=b.
\]
Thus $\NL_{\R}(\rho)$ has Zariski tangent dimension $b$ and
contains the smooth $b$-dimensional germ supplied by the
fixed-holonomy deformation space.  It follows that the
Noether--Lefschetz germ is itself smooth of dimension $b$ and coincides
with that germ.

\medskip

We prove closedness.  Suppose
\(t_j=[X_j,\phi_j]\to t=[X,\phi]\) in \(\Tg\), and each \(X_j\)
admits a \(\bar\rho\)-equivariant holomorphic map
$F_j:\widetilde X_j\to\HH$.  Fix a compact fundamental region and a
finite generating set of $\pi_1(\Sigma_g)$.  Schwarz--Pick gives uniform
local bounds once the value at one point is controlled.  If the base
values escaped every compact subset of $\HH$, a subsequence would
converge to a point of $\partial\HH$ fixed by every generator,
contradicting nonelementarity.  A normal-family argument therefore produces a $\bar\rho$-equivariant
holomorphic limit $F:\widetilde X\to\HH$.  It is nonconstant, since a
constant equivariant map would give a common fixed point in $\HH$.
Hence $\NL_{\R}(\rho)$ is closed.

\medskip

It remains to prove germ rigidity.  By Serre duality, the annihilator of
\eqref{eq:tangent-NL} in $H^0(X,K_X^2)$ is
\[
 \operatorname{im}\!\left(
 H^0(X,K_X\mathcal N^{-1})
 \xrightarrow{\,dF\,}
 H^0(X,K_X^2)
 \right)
 =
 H^0(X,K_X^2(-D)).
\]
The line bundle $K_X\mathcal N^{-1}$ is globally generated.  Indeed,
\[
 \deg(K_X\mathcal N^{-1})
 =
 2g-2+e\ge2g,
\]
because $e=2\deg L>0$, hence $e\ge2$.  Equivalently, failure of
generation at $x\in X$ would give
$H^0(X,\mathcal N(x))\neq0$, whereas
$\deg\mathcal N(x)=1-e\le-1$.  Therefore the common zero divisor of
$H^0(X,K_X^2(-D))$ is exactly $D$.  Thus the germ, already through its
tangent space, recovers the branching divisor $D$.

Finally, a branched hyperbolic metric in the conformal class of \(X\)
with prescribed branching divisor \(D\) is unique; see
\cite{Troyanov91}. Hence the two developing maps differ by an
orientation-preserving isometry of \(\HH\), and their projective
holonomies are conjugate in \(\PSL_2(\R)\).
\end{proof}

\section{Algebraic lifting and extension}\label{sec:extension}

We now prepare the algebraic-geometric input needed for the proof of
Theorem~\ref{thm:main}. Suppose that
\(U\subset\NL_{\R}(\rho)\) is a positive-dimensional connected component
and that \(p(U)\) is closed algebraic. The stabilizer statement \eqref{eq:stabilizer-inclusion} below is a rank-two phenomenon;
the corresponding higher-rank statement need not hold.

\begin{lemma}[Algebraic lifting lemma]\label{lem:lifting}
After passing to a finite level cover \(\Mg[\ell]\to\Mg\), there exist a smooth
connected closed algebraic subvariety \(S\subset\Mg[\ell]\) of
dimension \(b\) and a connected component \(\widetilde S\) of its inverse
image in \(\Tg\) such that
\[
 \widetilde S\subset U,
 \qquad
 \widetilde S\longrightarrow S
\]
is a covering map. The classifying map \(S\to\Mg\) is immersive on a
nonempty Zariski-open subset.
\end{lemma}

\begin{proof}
Choose a torsion-free level \(\ell\ge3\), so that
\[
q_\ell:\Tg\longrightarrow\Mg[\ell]
\]
is a covering map and \(\Mg[\ell]\) carries a universal curve.  The inverse
image of \(p(U)\) is closed algebraic.  Since \(q_\ell\) is locally
biholomorphic and \(U\) is a smooth \(b\)-fold by
Lemma~\ref{lem:local}, at a general point \(x\in U\) the germ of
\(q_\ell(U)\) coincides with the germ of the unique \(b\)-dimensional
irreducible component \(Z\) of this inverse image through \(q_\ell(x)\).

Choose a connected smooth Zariski-open subset \(S\subset Z\) containing
\(q_\ell(x)\) and avoiding the other components, and let \(\widetilde S\)
be the component of \(q_\ell^{-1}(S)\) containing \(x\).  Then
\(\widetilde S\cap U\) is open in \(\widetilde S\) by the preceding
local identification and closed by Lemma~\ref{lem:local}.  Since
\(\widetilde S\) is connected, \(\widetilde S\subset U\).  Thus
\(\widetilde S\to S\) is a covering.  After shrinking \(S\) away from
the ramification locus of \(\Mg[\ell]\to\Mg\), the classifying map
\(S\to\Mg\) is immersive.
\end{proof}

\begin{proposition}[Extension over an algebraic component]\label{prop:extension}
For \(S\) as above:
\begin{enumerate}
\item We have
\begin{equation}\label{eq:stabilizer-inclusion}
\operatorname{im}\!\left(\pi_1(S)\longrightarrow\operatorname{MCG}_g\right)
 \subset \operatorname{Stab}_{\operatorname{MCG}_g}([\bar\rho]).
\end{equation}
\item After replacing \(S\) by a finite algebraic cover and shrinking it,
there exist a smooth projective family
\[
 q:\mathcal C\longrightarrow S
\]
of genus-\(g\) curves and a representation
\begin{equation}\label{eq:extension-R}
 R:\pi_1(\mathcal C)\longrightarrow\SL_2(\R)
\end{equation}
whose restriction to every fiber is conjugate to the original
representation \(\rho\). The classifying map \(S\to\Mg\) remains immersive
on a nonempty open subset.
\end{enumerate}
\end{proposition}

\begin{proof}
For (i), set
\[
 \Gamma_S:=\operatorname{im}\!\left(\pi_1(S)\longrightarrow
 \operatorname{MCG}_g\right).
\]
By construction, \(\widetilde S\) is a connected component of
\(q_\ell^{-1}(S)\), and \(\Gamma_S\) is precisely the geometric
monodromy subgroup preserving this component. Thus
\begin{equation}\label{eq:gamma-preserves-lift}
 \gamma\widetilde S=\widetilde S,
 \qquad \gamma\in\Gamma_S.
\end{equation}
Moreover \(\widetilde S\subset U\subset\NL_{\R}(\rho)\), and
\(\dim\widetilde S=b\). The mapping class group transports
fixed-representation loci by
\[
 \gamma\bigl(\NL_{\R}(\rho)\bigr)
 =\NL_{\R}(\gamma\!\cdot\!\rho).
\]
Hence \eqref{eq:gamma-preserves-lift} gives
\[
 \widetilde S\subset\NL_{\R}(\rho)
 \cap\NL_{\R}(\gamma\!\cdot\!\rho).
\]
The representation \(\gamma\!\cdot\!\rho\) is again non-unitary and has
the same signed Euler number as \(\rho\). By Lemma~\ref{lem:local}, both
fixed-representation loci are smooth of dimension \(b\). Therefore, at a
general \(x\in\widetilde S\), the \(b\)-dimensional germ of
\(\widetilde S\) is the full germ of both loci. Germ rigidity in
Lemma~\ref{lem:local} then gives
\[
 [\gamma\!\cdot\!\bar\rho]=[\bar\rho].
\]
Since \(\gamma\in\Gamma_S\) was arbitrary, this proves
\eqref{eq:stabilizer-inclusion}.

For (ii), pull back the universal curve over \(\Mg[\ell]\) to \(S\).
After a generically finite base change and shrinking, an algebraic
multisection becomes a section, so
\[
 1\longrightarrow\pi_1(\Sigma_g)\longrightarrow\pi_1(\mathcal C)
 \longrightarrow\pi_1(S)\longrightarrow1
\]
splits. For \(a\in\pi_1(S)\), let
\(\alpha_a\in\Aut(\pi_1(\Sigma_g))\) be the induced action. By (i), there
is a unique \(A_a\in\PSL_2(\R)\) such that
\begin{equation*}
 A_a\bar\rho(\gamma)A_a^{-1}=\bar\rho(\alpha_a(\gamma)),
 \qquad \gamma\in\pi_1(\Sigma_g).
\end{equation*}
Uniqueness follows from the trivial centralizer of the nonelementary group
\(\bar\rho(\pi_1(\Sigma_g))\), and immediately gives
\(A_{ab}=A_aA_b\). Thus \(a\mapsto A_a\) defines
\[
 A:\pi_1(S)\longrightarrow\PSL_2(\R).
\]

The determinant-one lifts of \(\bar\rho\), up to conjugacy, form a finite
torsor under \(H^1(\Sigma_g,\Z/2)\). After a finite \'etale cover, we may
assume that the conjugacy class of \(\rho\) is fixed by \(\pi_1(S)\), so
each \(A_a\) has a lift intertwining \(\rho\) with
\(\rho\circ\alpha_a\). Put \(H=A(\pi_1(S))\), and let
\(\widetilde H\subset\SL_2(\R)\) be its full inverse image. Since
\(\widetilde H\) is a finitely generated linear group, residual finiteness
gives a finite-index subgroup not containing \(-I\); its projection is
therefore an isomorphism onto a finite-index subgroup of \(H\). After one
further finite \'etale cover, \(A\) admits a homomorphic lift
\[
 \widetilde A:\pi_1(S)\longrightarrow\SL_2(\R).
\]
Using the splitting, define
\[
 R(\gamma a):=\rho(\gamma)\widetilde A(a),
 \qquad \gamma\in\pi_1(\Sigma_g),\ a\in\pi_1(S).
\]
The intertwining relation shows that \(R\) is a homomorphism whose
restriction to every fiber is conjugate to \(\rho\). The finite base
changes preserve generic immersivity of the classifying map.
\end{proof}

\section{Proof of Theorem~\ref{thm:main}}\label{sec:main-proof}

We first fix the terminology used in this section. Following Simpson
\cite[\S4]{Simpson91} and Corlette--Simpson \cite[\S3]{CorletteSimpson08},
an \textbf{orbicurve} is an effective smooth one-dimensional
Deligne--Mumford stack with trivial generic stabilizer. Let \(B\) be a
compact orbicurve whose coarse curve has genus \(h\), and let
\(p_1,\ldots,p_n\) be its orbifold points with stabilizer orders
\(m_1,\ldots,m_n\). Its orbifold Euler characteristic is
\begin{equation*}
 \chi_{\orb}(B)
 :=2-2h-\sum_{j=1}^n\left(1-\frac1{m_j}\right).
\end{equation*}
Equivalently, if \(\pi:C\to B\) is a torsion-free finite orbifold cover of
degree \(m\), then
\[
 \chi(C)=m\,\chi_{\orb}(B).
\]
We refer to Scott \cite[\S2]{Scott83} for the surface-orbifold covering
picture and this definition. We use orbifold degrees for line bundles on
\(B\), and call \(B\) \textbf{hyperbolic} if \(\chi_{\orb}(B)<0\).

For a representation
\[
 \tau:\pi_1^{\orb}(B)\longrightarrow\PSL_2(\R)
\]
of a compact hyperbolic orbicurve, define its \textbf{orbifold Euler number}
by
\begin{equation*}
 e_B(\tau)
 :=\frac1m\,e\!\left(\tau|_{\pi_1(C)}\right),
\end{equation*}
where \(\pi:C\to B\) is any torsion-free finite orbifold cover of degree
\(m\), and the Euler number on the right is taken with the sign convention
of \eqref{eq:euler}. This is independent of the chosen cover, by passage to
a common finite refinement and multiplicativity of the Euler class under
finite covers. More generally, if \(f:X\to B\) is a nonconstant
representable holomorphic map of degree \(d\) from a smooth compact curve,
then naturality of the Euler class gives
\begin{equation}\label{eq:Euler-pullback-degree}
 e(\tau\circ f_*)=d\,e_B(\tau).
\end{equation}

We will use the following orbifold form of the Milnor--Wood inequality and
its equality case. The surface case is due to Milnor and Wood
\cite{Milnor58,Wood71}, while the characterization of equality for
\(\PSL_2(\R)\)-representations is due to Goldman
\cite[Corollary~C]{Goldman88}.

\begin{lemma}[Orbifold Milnor--Wood and maximality]\label{lem:orbifold-MW}
Let \(B\) be a compact hyperbolic orbicurve and let
\(\tau:\pi_1^{\orb}(B)\to\PSL_2(\R)\). Then
\begin{equation}\label{eq:orbifold-MW}
 |e_B(\tau)|\le -\chi_{\orb}(B).
\end{equation}
If equality holds, then \(\tau(\pi_1^{\orb}(B))\) is discrete.
\end{lemma}

\begin{proof}
Choose a torsion-free finite orbifold cover \(\pi:C\to B\) of degree
\(m\). Then
\[
 |e(\tau|_{\pi_1(C)})|
 \le -\chi(C)
 =-m\chi_{\orb}(B)
\]
by the Milnor--Wood inequality, which proves \eqref{eq:orbifold-MW} after
dividing by \(m\). If equality holds, Goldman’s maximality theorem implies
that \(\tau|_{\pi_1(C)}\) is an isomorphism onto a discrete subgroup of
\(\PSL_2(\R)\). Its image has finite index in
\(\tau(\pi_1^{\orb}(B))\), so the latter is discrete as well.
\end{proof}

\begin{definition}[Marked fixed-target orbifold Hurwitz component]
\label{def:fixed-target-hurwitz}
Let \(B\) be a compact hyperbolic orbicurve and \(d>0\). Denote by
\(\cH_{g,d}(B)\) the stack of representable degree-\(d\) orbifold maps
\(f:X\to B\) from smooth connected genus-\(g\) curves, and set
\[\widetilde{\cH}_{g,d}(B):=\cH_{g,d}(B)\times_{\Mg}\Tg,\] where
\(\cH_{g,d}(B)\to\Mg\) forgets the map to \(B\). 

A connected submanifold
\(U\subset\Tg\) is a \textbf{marked fixed-target orbifold Hurwitz component}
if, for some connected component \(\cH^\circ\subset\cH_{g,d}(B)\), a
connected component
\(\widetilde{\cH}^{\circ}\subset\cH^\circ\times_{\Mg}\Tg\) maps
isomorphically onto \(U\) under the forgetful map to \(\Tg\).
\end{definition}

\subsection{Algebraicity implies discreteness}\label{sec:alg-to-disc}

We use the following rank-two factorization consequence of
Corlette--Simpson \cite{CorletteSimpson08} and Loray--Pereira--Touzet
\cite{LPT}.

\begin{theorem}
\label{thm:CS-LPT-alternative}
Let \(Z\) be a smooth connected complex quasi-projective variety and let
\(R:\pi_1(Z)\to\SL_2(\C)\) be Zariski dense. Write \(\bar R\) for its
projectivization. Then one of the following holds projectively:
\begin{enumerate}
\item \(\bar R\) factors through an algebraic morphism \(Z\to B\) to an
algebraic orbicurve;
\item \(\bar R\) belongs to the polydisk Shimura case: there exist a totally
imaginary number field \(L\), a rank-two projective \(\mathcal O_L\)-module
with polarizing form, and an algebraic map
\[
\Psi:Z\longrightarrow\mathscr A
\]
to the associated polydisk Shimura modular Deligne--Mumford stack such that
the corresponding projective local system is induced by the tautological
one.
\end{enumerate}
\end{theorem}

\begin{proof}
Loray--Pereira--Touzet \cite[Corollary~B]{LPT} gives the projective
alternative between an orbicurve and a polydisk Shimura modular target
for every non-virtually-abelian rank-two representation; the present
Zariski-dense representation is certainly non-virtually abelian.  In the
polydisk case we only use the resulting projective Shimura factorization.
Factorization through a Deligne--Mumford curve is equivalent, after
projectivization, to factorization through an orbicurve by
\cite[Corollary~3.3]{CorletteSimpson08}.
\end{proof}

\begin{lemma}
\label{lem:real-polydisk}
In the situation of Proposition~\ref{prop:extension}, suppose that
Theorem~\ref{thm:CS-LPT-alternative} gives the polydisk alternative. Then
the projective representation \(\bar R\) factors through an algebraic
orbicurve.
\end{lemma}

\begin{proof}
Let $\eta:L\hookrightarrow\C$ be the distinguished embedding through which the projective
representation $\bar R$ factors. Thus
\[
 \bar R\simeq\Psi^*\mathbb P(\mathcal V^\eta).
\]

Let $\tau_1=\eta,\tau_2,\ldots,\tau_m$ be all embeddings corresponding to the noncompact factors of the
polydisk. For a general fiber \(X=X_s\) of $\mathcal C$, let \(V_i\) denote the
restriction to \(X\) of \(\Psi^*\mathcal V^{\tau_i}\). Thus \(i=1\)
is the factor corresponding to \(\bar\rho\).

For every noncompact factor, the corresponding conjugate rank-two
local system is a non-unitary variation of Hodge structure; see
Corlette--Simpson \cite[\S9--10]{CorletteSimpson08}. Hence its period
map $F_i:\widetilde X\longrightarrow\HH$ is nonconstant.

Let \(L_i\subset V_i\) be its positive Hodge line, and put
\[
 \mathcal N_i
 :=
 \operatorname{Hom}(L_i,V_i/L_i)
 =
 (\det V_i)\otimes L_i^{-2}.
\]The flat determinant has degree zero, so curvature gives
\[
 e_i:=-\deg \mathcal N_i=2\deg L_i\ge2,
 \qquad H^0(X,\mathcal N_i)=0.
\]
Let \(D_i=\Div(dF_i)\), and let \(\NL_i\) denote the local
fixed-representation locus through \([X]\) associated with the \(i\)-th
factor. Then \(\mathcal N_i\simeq T_X(D_i)\), and the similar deformation-complex calculation as in Lemma~\ref{lem:local} gives
\[
 T_{[X]}\NL_i=\ker\bigl(H^1(T_X)\to H^1(\mathcal N_i)\bigr),\qquad
 \operatorname{Ann}(T_{[X]}\NL_i)=H^0(X,K_X^2(-D_i)).
\]
Moreover \(\deg(K_X\mathcal N_i^{-1})=2g-2+e_i\ge2g\), so $K_X\mathcal N_i^{-1}$ is
globally generated and the common zero divisor of the displayed
annihilator is exactly \(D_i\). Since \(S\) is the full local
Hodge component for the first factor and every tautological factor is a
VHS on the total space \(\mathcal C\),
\[
 T_sS=T_{[X]}\NL_1\subset T_{[X]}\NL_i.
\]
Taking annihilators therefore gives
\[
 H^0(X,K_X^2(-D_i))\subset H^0(X,K_X^2(-D_1)),
\]
and hence
\begin{equation}\label{eq:real-Dcontain}
 D_1\le D_i.
\end{equation}

Write \(\omega_i:T_{\mathcal C}|_X\to \mathcal N_i\) for the \(i\)-th coordinate of
the differential of \(\Psi\). By \eqref{eq:real-Dcontain}, division of the
vertical derivatives defines a holomorphic bundle map
\[
 a_i:=\frac{dF_i}{dF_1}:\mathcal N_1\longrightarrow \mathcal N_i.
\]
The difference \(\omega_i-a_i\omega_1\) vanishes on \(T_X\).  Using
\[
 0\longrightarrow T_X\longrightarrow T_{\mathcal C}|_X
 \longrightarrow T_sS\otimes\mathcal O_X\longrightarrow0,
\]
it therefore descends to a map
\(T_sS\otimes\mathcal O_X\to \mathcal N_i\), which is zero because
\(H^0(X,\mathcal N_i)=0\). Thus \(\omega_i=a_i\omega_1\) for every noncompact factor;
compact factors have zero-dimensional period domain.  Hence \(d\Psi\) has
generic rank one.
Therefore the algebraic image of \(\Psi\) is a curve.

Let \(\mathcal B\) be the normalization of the one-dimensional
Deligne--Mumford schematic image of \(\Psi\). Since \(\mathcal C\) is
normal, \(\Psi\) factors through \(\mathcal B\). Consequently, the
tautological rank-two local system \(\Psi^*\mathcal V^\eta\) factors through the Deligne--Mumford curve \(\mathcal B\). Corlette--Simpson \cite[Corollary~3.3]{CorletteSimpson08} then implies
that \(\bar R\) factors projectively through an algebraic orbicurve.
\end{proof}

Assume that \(p(U)\) is closed
algebraic. We prove \((i)\Rightarrow(ii)\) of Theorem~\ref{thm:main} in four steps.  

\medskip

\noindent\textbf{Step 1: Arithmetic reduction and orbicurve factorization.} By Proposition~\ref{prop:extension}, after 
base changes we obtain a smooth projective family
\(q:\mathcal C\to S\) and a total representation
\[
 R:\pi_1(\mathcal C)\longrightarrow\SL_2(\R)
\]
restricting to \(\rho\) on every fiber. Since the fiber image is
nonelementary, \(R\) is Zariski dense in \(\SL_2(\C)\).
Theorem~\ref{thm:CS-LPT-alternative} now gives two cases. In the orbicurve
case we obtain directly an algebraic factorization
\[
 \mathfrak f:\mathcal C\longrightarrow B,
 \qquad
 \bar R=\sigma\circ\mathfrak f_*.
\]
In the polydisk case, Lemma~\ref{lem:real-polydisk} gives a factorization
of the same form. Hence, in either case, we may fix such an orbicurve
factorization and proceed to Steps~2--4 below.

\medskip
\noindent\textbf{Step 2: Geometry of the orbicurve factor and the Euler-number bound.}
For \(s\in S\), write \(X_s:=q^{-1}(s)\) and
\(f_s:=\mathfrak f|_{X_s}:X_s\to B\). Each \(f_s\) is nonconstant;
otherwise \(\bar\rho(\pi_1(\Sigma_g))\) would be contained in the image
under \(\sigma\) of a finite orbifold stabilizer, contradicting
nonelementarity. Thus \(f_s\) is surjective and \(B\) is compact. Moreover,
\(B\) is hyperbolic, since \(\pi_1^{\orb}(B)\) has a nonelementary image.

The subgroup \(\mathfrak f_*(\pi_1(\mathcal C))\) has finite index in
\(\pi_1^{\orb}(B)\). Indeed, it contains
\((f_s)_*(\pi_1(X_s))\). If \(B'\to B\) is the connected orbifold cover
corresponding to this latter subgroup, then \(f_s\) lifts to a nonconstant
map \(X_s\to B'\). Its image is open, and it is compact because \(X_s\)
is compact; hence it is also closed. Since \(B'\) is connected, the lift is
surjective, so \(B'\) is compact and therefore \(B'\to B\) has finite
degree. Passing to the finite orbifold cover corresponding to
\(\mathfrak f_*(\pi_1(\mathcal C))\), the map \(\mathfrak f\) lifts.
After one global conjugation, the restriction of \(\sigma\) to this
subgroup equals the real representation \(\bar R\); hence on the new
orbicurve \(\sigma\) takes values in \(\PSL_2(\R)\). We keep the notation
\(B,\mathfrak f,\sigma\) and put
\[
 \kappa:=-\chi_{\orb}(B)>0.
\]

Let \(d:=\deg f_s\), which is constant on the connected base \(S\).
Orbifold Riemann--Hurwitz gives
\begin{equation}\label{eq:orb-RH-seq}
0\longrightarrow T_{X_s}\xrightarrow{df_s}f_s^*T_B
\longrightarrow\mathcal Q_s\longrightarrow0,
\end{equation}
where
\begin{equation}\label{eq:r-orbifold-Hurwitz}
 r:=\operatorname{length}\mathcal Q_s=2g-2-d\kappa.
\end{equation}
Since \(\bar\rho=\sigma\circ(f_s)_*\), the naturality formula
\eqref{eq:Euler-pullback-degree} gives
\begin{equation}\label{eq:euler-factorization}
 e(\bar\rho)=d\,e_B(\sigma).
\end{equation}
In particular \(e_B(\sigma)>0\). By Lemma~\ref{lem:orbifold-MW},
\(e_B(\sigma)\le\kappa\), and therefore
\begin{equation}\label{eq:b-ge-r}
 b:=2g-2-e(\bar\rho)
   =2g-2-d\,e_B(\sigma)
   \ge 2g-2-d\kappa
   =r.
\end{equation}

\medskip
\noindent\textbf{Step 3: Dimension comparison.}
The complex
\[
C_{f_s}^{\bullet}:=[T_{X_s}\xrightarrow{df_s}f_s^{*}T_B]
\]
governs deformations of \(f_s:X_s\to B\) with the target \(B\) fixed. By \eqref{eq:orb-RH-seq}, \(C_{f_s}^{\bullet}\simeq \mathcal Q_s[-1]\). Since \(\mathcal Q_s\) is zero-dimensional, we have $\mathbb H^2(X_s,C_{f_s}^{\bullet})=0,$ so these deformations are unobstructed, and
\[
\dim \mathbb H^1(X_s,C_{f_s}^{\bullet})
=h^0(X_s,\mathcal Q_s)
=\operatorname{length}(\mathcal Q_s)=r.
\]
Moreover, \(H^0(X_s,f_s^{*}T_B)=0\), since \(\deg f_s^{*}T_B=-d\kappa<0\); hence the forgetful map to \(\Mg\) is locally immersive. As the
classifying map \(S\to\Mg\) is generically immersive and \(\dim S=b\),
we obtain \(b\le r\). Together with \eqref{eq:b-ge-r}, this gives $b=r$.

Comparing \eqref{eq:r-orbifold-Hurwitz} with
\eqref{eq:euler-factorization} and using \(b=2g-2-e(\bar\rho)\), we obtain
\begin{equation}\label{eq:maximal-orbifold-Euler}
 e_B(\sigma)=\kappa=-\chi_{\orb}(B).
\end{equation}

\medskip
\noindent\textbf{Step 4: Maximality and discreteness.}
Equation \eqref{eq:maximal-orbifold-Euler} is the equality case of the
orbifold Milnor--Wood inequality. Lemma~\ref{lem:orbifold-MW} therefore
implies that
\[
 \sigma(\pi_1^{\orb}(B))\subset\PSL_2(\R)
\]
is discrete. Since
\[
 \bar\rho(\pi_1(\Sigma_g))
 \subset\sigma(\pi_1^{\orb}(B)),
\]
the image of \(\bar\rho\) in \(\PSL_2(\R)\) is discrete. This proves
\((i)\Rightarrow(ii)\).

\subsection{Discreteness, Hurwitz realization, and algebraicity}
\label{sec:disc-to-alg}

Assume that \(\rho(\pi_1(\Sigma_g))\subset\SL_2(\R)\) is discrete, and set
\(\Gamma:=\bar\rho(\pi_1(\Sigma_g))\). We prove \((ii)\Rightarrow(iii)\Rightarrow(i)\) of Theorem~\ref{thm:main}.

\medskip
\noindent\textbf{Step 1: The fixed hyperbolic orbicurve and the degree.}
Set
\[
B:=\Gamma\backslash\HH.
\]
The quotient map $\HH\longrightarrow B$ is the orbifold universal covering, and hence induces the uniformizing
representation
\[
\sigma_B:\pi_1^{\orb}(B)\simeq\Gamma
\hookrightarrow\PSL_2(\R).
\]
For \(t=[X_t,\phi_t]\in U\), let
\(F_t:\widetilde X_t\to\HH\) be the
\(\bar\rho\)-equivariant period map. It descends to a nonconstant orbifold
holomorphic map
\[
f_t:X_t\to B
\]
satisfying $\bar\rho=\sigma_B\circ(f_t)_* .$

Since \(X_t\) is compact, \(f_t\) is
surjective and \(B\) is compact. Put \(\kappa:=-\chi_{\orb}(B)>0\). As
\(B\) is uniformized by \(\HH\), the period tangent line is
\(f_t^*T_B\); hence
\[e(\bar\rho)=-\deg f_t^*T_B=(\deg f_t)\kappa.\] Thus
\(d:=\deg f_t=e(\bar\rho)/\kappa\) is independent of \(t\).

Let \(\mathcal H_{g,d}(B)\) be as in Definition~\ref{def:fixed-target-hurwitz}. It is the smooth-source locus in the stack of twisted stable maps to \(B\) \cite{AbramovichVistoli02}.

\medskip
\noindent\textbf{Step 2: Finiteness of the fixed-target Hurwitz morphism.}
Let \(\pi_B:\cH_{g,d}(B)\to\Mg\) be the forgetful morphism. The
fixed-degree smooth-source locus is of finite type, being an open substack
of the corresponding stack of twisted stable maps. For
\([f:X\to B]\in\cH_{g,d}(B)\), one has
\(\deg f^*T_B=-d\kappa<0\), hence \(H^0(X,f^*T_B)=0\). Thus the fiber of
\(\pi_B\) is zero-dimensional at every point, and \(\pi_B\) is
quasi-finite.

For properness, we use the valuative criterion. Let \(\mathscr R\) be a DVR with fraction field \(K\), and suppose that a smooth family of genus-\(g\) curves
\[
X_{\mathscr R}\longrightarrow \operatorname{Spec} \mathscr R
\]
is given together with a generic-fiber map
\[
f_K:X_K\longrightarrow B
\]
representing a \(K\)-point of \(\mathcal H_{g,d}(B)\). It suffices to show, after a finite base change of \(\mathscr R\), that \(f_K\) extends over the special point with smooth source \(X_0\).

By properness of the stack of twisted stable maps \cite[Theorem~1.4.1]{AbramovichVistoli02}, after such a finite base change \(f_K\) extends to a twisted stable map
\[
 f:\mathfrak X_{\mathscr R}\longrightarrow B.
\]
The coarse space of its special fiber is a nodal curve. By separatedness of \(\overline{\mathcal M}_g\), its stabilization is isomorphic to \(X_0\). Hence every component removed by stabilization belongs to a connected tree of rational components contracted to a point of \(X_0\); we call such a subcurve an \textbf{exceptional rational tree}.

Suppose that such a tree exists, and choose an \textbf{outermost leaf} \(E\), i.e. a component corresponding to a leaf of the dual tree, so that \(E\) meets the rest of the curve in exactly one point. If \(f|_E\) is nonconstant, its differential gives a nonzero morphism
\[
T_E^{\orb}\longrightarrow (f|_E)^*T_B,
\]
which is impossible since $\deg T_E^{\orb}>0>\deg (f|_E)^*T_B.$
Indeed, if the unique attaching point has stabilizer order \(m\), then
\(\deg T_E^{\orb}=2-(1-1/m)=1+1/m>0\) (with \(m=1\) in the untwisted
case).

If \(f|_E\) is constant, then \(E\) has only one special point and therefore violates the stability condition for a contracted rational component. Thus no outermost leaf can occur, and hence no exceptional rational tree exists.

It follows that the special source is already \(X_0\), so the generic map
\(f_K\) extends to a map \(X_0\to B\), and therefore to an
\(\mathscr R\)-point of \(\mathcal H_{g,d}(B)\). Uniqueness follows from
separatedness of the stack of twisted stable maps. Hence \(\pi_B\) is
proper. The forgetful morphism is representable on this smooth-source
locus: an automorphism of a map inducing the identity on its source is
trivial. Since \(\pi_B\) is also quasi-finite, it is finite.

\medskip
\noindent\textbf{Step 3: Identification with the Noether--Lefschetz component.}
Fix \(t_0=[X_0,\phi_0]\in U\) and set \(f_0:=f_{t_0}\). Let
\(\cH^\circ\subset\cH_{g,d}(B)\) be the connected component containing
\([f_0]\), and let
\(\widetilde{\cH}^{\circ}\subset\cH^\circ\times_{\Mg}\Tg\) be the
component containing the marked point \((f_0,t_0)\).

On \(\widetilde{\cH}^{\circ}\), the induced marked homomorphism
\(\pi_1(\Sigma_g)\to\Gamma\) is locally constant and hence equals
\(\bar\rho\). Thus the lifted maps to \(\HH\) are
\(\bar\rho\)-equivariant and define the positive Hodge line on the marked local system determined by \(\rho\); consequently
\[\widetilde{\cH}^{\circ}\subset U.\]
 
 Conversely, for every \(t\in U\), the period map descends uniquely to
\(f_t:X_t\to B\). The local fixed-holonomy description in
Lemma~\ref{lem:local} gives, near every point of \(U\), a holomorphic lift
to the marked fixed-target map space; uniqueness makes these local lifts
agree on overlaps. They therefore glue over the connected manifold \(U\),
and their image lies in the component containing \((f_0,t_0)\). Hence the
forgetful map \(\widetilde{\cH}^{\circ}\to U\) is bijective.

For any \([f:X\to B]\in\cH^\circ\), the complex
\([T_X\to f^*T_B]\) has tangent dimension
\[
r=2g-2-d\kappa=2g-2-e(\bar\rho)=b,
\]
and no obstructions, while \(H^0(X,f^*T_B)=0\). Thus the local
dimension of the Hurwitz space equals \(b\), and the forgetful map to
\(\Tg\) is a local biholomorphism.
The preceding bijection therefore gives
\begin{equation}\label{eq:Hurwitz-NL-identification}
\widetilde{\cH}^{\circ}\xrightarrow{\sim}U.
\end{equation}
This proves \((ii)\Rightarrow(iii)\).

\medskip
\noindent\textbf{Step 4: Algebraicity and closedness.}
By Step~2, \(\pi_B\) is finite. Since \(\cH^\circ\) is a connected
component of a finite-type stack, it is open and closed; by
\eqref{eq:Hurwitz-NL-identification},
\(p(U)=\pi_B(\cH^\circ)\). Hence \(p(U)\) is a closed algebraic subvariety
of \(\Mg\). This proves \((iii)\Rightarrow(i)\) and completes the proof of
Theorem~\ref{thm:main}.\qed

\section{Two slit examples}\label{sec:examples}

We finish by exhibiting the two alternatives in
Theorem~\ref{thm:main}.  Both examples arise from the same slit
surgery. 

\subsection{The slit surgery}\label{subsec:slit-surgery}

We first record the local construction common to both examples.  Let
\(\Gamma\subset\SL_2(\R)\) be torsion-free and cocompact, let
\[
 Y:=\Gamma\backslash\HH,
 \qquad
 \pi_Y:\HH\longrightarrow Y,
\]
and let \(A\in\SL_2(\R)\) be such that the projective transformation
\(\bar A\) does not belong to the projective image of \(\Gamma\).  Equivalently,
\(\pi_Y\) and \(\pi_Y\circ A\) are not identical.  Choose
\(\zeta_0\in\HH\) with
\[
 \pi_Y(\zeta_0)\neq\pi_Y(A\zeta_0).
\]
After shrinking around \(\zeta_0\), choose a disk \(\mathbb D\Subset\HH\) such that
both restrictions
\[
 \pi_Y|_{\mathbb D}:\mathbb D\longrightarrow \pi_Y(\mathbb D),
 \qquad
 \pi_Y|_{A(\mathbb D)}:A(\mathbb D)\longrightarrow \pi_Y(A(\mathbb D))
\]
are biholomorphic and the two image disks in \(Y\) are disjoint.  Fix a
coordinate
\[
 \chi:\mathbb D\longrightarrow\Delta_R:=\{|t|<R\},
 \qquad
 \chi(\zeta_0)=0.
\]
We use the coordinate \(t=\chi(\zeta)\) on \(\pi_Y(\mathbb D)\), and the transported
coordinate \(t=\chi(A^{-1}\zeta)\) on \(\pi_Y(A(\mathbb D))\).

Choose $r_0$ with $0<r<r_0<R$ and set
\[
 \mathcal P:=\{(z,w)\in\Delta_r\times\Delta_r:z\neq w\}.
\]
For \((z,w)\in\mathcal P\), consider the connected double cover
\begin{equation}\label{eq:local-cover}
 C_{z,w}:=
 \bigl\{(t,\eta):|t|<R,\ \eta^2=(t-z)(t-w)\bigr\}
 \longrightarrow \Delta_R,
 \qquad
 (t,\eta)\longmapsto t.
\end{equation}
It is simply ramified over \(t=z\) and \(t=w\).  Over the outer annulus
\(r_0<|t|<R\), which contains neither branch point, the cover splits into
two holomorphic sheets
\begin{equation}\label{eq:q-branches}
 \eta=\pm q_{z,w}(t),
 \qquad
 q_{z,w}(t):=
 t\sqrt{1-z/t}\sqrt{1-w/t},
\end{equation}
where both square roots are the branches near \(1\).  Remove from the two
coordinate disks in \(Y\) the closed subdisks \(|t|\le r_0\).  We then glue
the sheet \(\eta=+q_{z,w}(t)\) to the collar in \(\pi_Y(\mathbb D)\), and the sheet
\(\eta=-q_{z,w}(t)\) to the collar in \(\pi_Y(A(\mathbb D))\), in both cases by
matching the same coordinate \(t\).  In terms of the hyperbolic charts in
\(\HH\), the transition on the first collar is the identity and the
transition on the second collar is \(A\).  Thus the local branched chart
\[
 C_{z,w}\longrightarrow \mathbb D\subset\HH,
 \qquad
 (t,\eta)\longmapsto \chi^{-1}(t),
\]
fits with the original hyperbolic structure on the complement of the two
disks.

The double cover \(C_{z,w}\to\Delta_R\) is topologically an annulus: it is
connected, its Euler characteristic is
\(2\chi(\Delta_R)-2=0\), and the inverse image of the boundary has two
components.  Hence the above operation removes two disks from \(Y\) and
inserts an annulus.  It therefore adds one handle.  In particular, when
\(g(Y)=g-1\), the resulting compact Riemann surface, denoted \(X_{z,w}\),
has genus \(g\).  The two ramification points of
\eqref{eq:local-cover} become precisely the two simple branch points of the
branched hyperbolic structure on \(X_{z,w}\).

Finally, both the equation \eqref{eq:local-cover} and the collar
identifications \eqref{eq:q-branches} depend holomorphically on
\((z,w)\in\mathcal P\).  They therefore produce a proper holomorphic
submersion
\begin{equation*}
 \varpi:\mathscr X\longrightarrow\mathcal P,
 \qquad
 \varpi^{-1}(z,w)=X_{z,w},
\end{equation*}
which we call the \textbf{slit family}.  With standard generators adapted to
the added handle, the old generators retain their holonomy in \(\Gamma\),
while the two new handle generators may be chosen to have projective
holonomy \(\bar A\) and the identity, respectively.

\subsection{A non-discrete example and a non-algebraic Noether--Lefschetz locus}

We prove Proposition~\ref{prop:non-discrete-example}.  Fix \(g\ge3\) and set
\(h:=g-1\).  By Takeuchi \cite{Takeuchi71}, there is a torsion-free
cocompact Fuchsian group
\[
 \Gamma_0\subset\SL_2(\Q)
\]
such that
\[
 Y:=\Gamma_0\backslash\HH
\]
is a compact Riemann surface of genus \(h\).

Consider
\begin{equation*}
 A:=\frac15
 \begin{pmatrix}
 3&-4\\
 4&3
 \end{pmatrix}
 \in\SL_2(\Q).
\end{equation*}
Then \(\bar A\) is elliptic since \(\operatorname{tr}(A)=6/5\). Moreover it has infinite projective order: if \(A^m=\pm I\) for some
\(m>0\), then its eigenvalues would be roots of unity and
\(\tr(A)=6/5\) would be an algebraic integer, hence an integer.  Thus the
cyclic subgroup generated by \(\bar A\) is non-discrete.

Since \(\Gamma_0\) is discrete, it contains no infinite-order elliptic
element.  Thus \(\bar A\) does not belong to the projective image of
\(\Gamma_0\), and the construction of
Section~\ref{subsec:slit-surgery} applies with \(\Gamma=\Gamma_0\).  We obtain the holomorphic slit family
\(\varpi:\mathscr X\to\mathcal P\) of genus-\(g\) curves \(X_{z,w}\).
Analytic continuation of the local hyperbolic charts gives a branched
developing map
\[
 F_{z,w}:\widetilde X_{z,w}\longrightarrow\HH
\]
whose transitions are in the group generated by \(\Gamma_0\) and \(A\).
By construction, its only critical points are the two simple branch points
lying over \(t=z\) and \(t=w\).

Choose standard surface-group generators
\[
 \pi_1(\Sigma_g)=
 \left\langle
 \alpha_1,\beta_1,\ldots,\alpha_{g-1},\beta_{g-1},a,b
 \ \middle|\
 \prod_{j=1}^{g-1}[\alpha_j,\beta_j][a,b]=1
 \right\rangle
\]
and generators \(G_j,H_j\in\Gamma_0\) for the genus-\((g-1)\) surface group.
Define
\begin{equation*}
 \rho_{\mathrm{nd}}(\alpha_j):=G_j,\qquad
 \rho_{\mathrm{nd}}(\beta_j):=H_j,\qquad
 \rho_{\mathrm{nd}}(a):=A,\qquad
 \rho_{\mathrm{nd}}(b):=I.
\end{equation*}
Let \(\phi_{z,w}:\Sigma_g\to X_{z,w}\) denote the marking determined by
these generators. The slit developing maps have projective holonomy
\(\bar\rho_{\mathrm{nd}}\).

Because the image contains the cocompact Fuchsian group \(\Gamma_0\),
\(\rho_{\mathrm{nd}}\) is absolutely irreducible and non-unitary.  It
preserves no full integral lattice: otherwise, after a rational change of
basis, \(A\) would lie in \(\SL_2(\Z)\), contradicting
\(\tr(A)=6/5\).

The equivariant holomorphic map \(F_{z,w}\) determines the Hodge line in
the rational local system defined by \(\rho_{\mathrm{nd}}\); the standard
alternating form on \(\Q^2\) gives the polarization.  Hence
\[
 [X_{z,w},\phi_{z,w}]\in\NL_{\Q}(\rho_{\mathrm{nd}}).
\]
Since \(F_{z,w}\) has two simple branch points, the pullback line bundle
\(F_{z,w}^*T_{\HH}\) on \(\widetilde X_{z,w}\) carries the
\(\pi_1(\Sigma_g)\)-equivariant structure
\[
 \gamma\cdot(x,v)
 :=\bigl(\gamma x,
 d\bar\rho_{\mathrm{nd}}(\gamma)_{F_{z,w}(x)}v\bigr).
\]
Its descent is the line bundle
\[
 \mathcal N_{z,w}
 :=\pi_1(\Sigma_g)\backslash F_{z,w}^*T_{\HH}
 \longrightarrow X_{z,w}.
\]
The differential descends to
\[
 dF_{z,w}:T_{X_{z,w}}\longrightarrow\mathcal N_{z,w}.
\]
Set \(D_{z,w}:=\Div(dF_{z,w})\).  The divisor \(D_{z,w}\) consists of the two simple branch points.
Hence
\[
 \mathcal N_{z,w}\simeq T_{X_{z,w}}(D_{z,w}),
 \qquad
 \deg\mathcal N_{z,w}
 =\deg T_{X_{z,w}}+\deg D_{z,w}
 =4-2g,
\]
and therefore $e(\bar\rho_{\mathrm{nd}})=2g-4.$ Lemma~\ref{lem:local} then shows that every connected component of
\(\NL_{\Q}(\rho_{\mathrm{nd}})\) is a smooth complex surface.

Finally, the image of \(\bar\rho_{\mathrm{nd}}\) in \(\PSL_2(\R)\) contains the infinite-order elliptic element \(\bar A\), whose powers accumulate at the identity.  Hence
\[
 \bar\rho_{\mathrm{nd}}\bigl(\pi_1(\Sigma_g)\bigr)
\]
is not discrete. This
proves Proposition~\ref{prop:non-discrete-example}.

\subsection{A discrete non-uniformizing example}

We now prove Proposition~\ref{prop:discrete-example}.  Fix \(g\ge4\) and
put
\[
 d:=g-2\ge2.
\]
Again by Takeuchi \cite{Takeuchi71}, choose a torsion-free cocompact
Fuchsian group
\[
 \Delta\subset\SL_2(\Q),
 \qquad
 C:=\Delta\backslash\HH,
 \qquad
 g(C)=2.
\]
Choose a surjection
\[
 \nu:\Delta\longrightarrow\Z/d\Z.
\]
Such a surjection exists because \(\Delta\simeq\pi_1(C)\) and
\(\Delta^{\mathrm{ab}}\simeq H_1(C,\Z)\simeq\Z^4\), which surjects onto
\(\Z/d\Z\).
Set
\[
 \Gamma_0:=\ker\nu,
\]
and choose \(A\in\Delta\) with \(\nu(A)=1\).  Then
\[
 \Gamma_0\triangleleft\Delta,\qquad
 [\Delta:\Gamma_0]=d,\qquad
 \langle\Gamma_0,A\rangle=\Delta.
\]
The quotient $Y:=\Gamma_0\backslash\HH$ is an unramified cyclic cover of \(C\) of degree \(d\).  Since \(g(C)=2\),
Riemann--Hurwitz gives
\[
 g(Y)-1=d,
 \qquad\text{hence}\qquad
 g(Y)=g-1.
\]

Since \(\nu(A)=1\), we have \(A\notin\Gamma_0\).  In fact
\(\bar A\) does not belong to the projective image of \(\Gamma_0\): otherwise
\(A=\pm\gamma\) for some \(\gamma\in\Gamma_0\), and the negative sign would
force \(-I\in\Delta\), contrary to torsion-freeness.  We may therefore apply
Section~\ref{subsec:slit-surgery} with \(\Gamma=\Gamma_0\) and this
element \(A\).  In choosing the disk \(\mathbb D\), shrink it further so that
its projection to \(C\) is injective.  Since \(A\in\Delta\), the two disks
\(\pi_Y(\mathbb D)\) and \(\pi_Y(A(\mathbb D))\) in \(Y\) project to the same coordinate
disk in \(C\).  Consequently the local projection
\((t,\eta)\mapsto t\) in \eqref{eq:local-cover} is compatible with the
covering \(Y\to C\) on both collars and therefore extends to a holomorphic
map
\begin{equation*}
 f_{z,w}:X_{z,w}\longrightarrow C
\end{equation*}
of degree \(d\) with exactly two simple ramification points.  Consequently
\[
 2g(X_{z,w})-2
 =d\bigl(2g(C)-2\bigr)+2
 =2d+2,
\]
so \(g(X_{z,w})=d+2=g\).

Using generators \(G_j,H_j\in\Gamma_0\) as above, define
\begin{equation*}
 \rho_{\mathrm{disc}}(\alpha_j):=G_j,\qquad
 \rho_{\mathrm{disc}}(\beta_j):=H_j,\qquad
 \rho_{\mathrm{disc}}(a):=A,\qquad
 \rho_{\mathrm{disc}}(b):=I.
\end{equation*}
Let \(\phi_{z,w}:\Sigma_g\to X_{z,w}\) denote the marking determined by
these generators. Its image is
\[
 \rho_{\mathrm{disc}}\bigl(\pi_1(\Sigma_g)\bigr)
 =\langle\Gamma_0,A\rangle=\Delta.
\]
Thus the monodromy is discrete and cocompact.  As before,
\(\rho_{\mathrm{disc}}\) is absolutely irreducible and non-unitary.

It is also non-integral.  If it preserved a full integral lattice in
\(\Q^2\), then
after a rational change of basis \(\Delta\) would lie in
\(\SL_2(\Z)\).  Projectively, the cocompact lattice \(\bar\Delta\) would
then be a lattice subgroup of \(\PSL_2(\Z)\), hence of finite index, which
is impossible.

Let
\[
\sigma:\pi_1(C)\simeq\Delta\hookrightarrow\SL_2(\Q)
\]
be the uniformizing rational representation of \(C\). The map
\(f_{z,w}\) shows that the local system defined by
\(\rho_{\mathrm{disc}}\) on \(X_{z,w}\) is the pullback of this
weight-one \(\mathbb Q\)VHS. Hence
\[
[X_{z,w},\phi_{z,w}]\in\NL_{\Q}(\rho_{\mathrm{disc}}).
\]
Moreover,
\[
e(\bar\rho_{\mathrm{disc}})
=d\bigl(2g(C)-2\bigr)
=2d
=2g-4.
\]
Thus Lemma~\ref{lem:local} gives a nonempty smooth two-dimensional
Noether--Lefschetz locus.

Finally, the VHS on \(X_{z,w}\) is not uniformizing.  Indeed, its period
map factors through the branched degree-\(d\) map $f_{z,w}: X_{z,w}\to C,$ and has two simple branch points.  Equivalently,
\[
 e(\bar\rho_{\mathrm{disc}})
 =2g-4
 <2g-2,
\]
whereas a uniformizing representation has maximal Euler number
\(2g-2\).  We have therefore produced a $\mathbb Q$PVHS with
discrete monodromy which is not uniformizing.  This proves
Proposition~\ref{prop:discrete-example}.


\begin{thebibliography}{99}

\bibitem{AbramovichVistoli02}
D.~Abramovich and A.~Vistoli,
\emph{Compactifying the space of stable maps},
J. Amer. Math. Soc. \textbf{15} (2002), no.~1, 27--75.

\bibitem{BaldiLam26}
G.~Baldi and Y.~H.~J.~Lam,
\emph{Murphy's law in non-abelian Hodge theory},
arXiv:2608.02875v1, 2026.

\bibitem{CDK}
E.~Cattani, P.~Deligne, and A.~Kaplan,
\emph{On the locus of Hodge classes},
J. Amer. Math. Soc. \textbf{8} (1995), no.~2, 483--506.

\bibitem{CTW}
B.~Collier, J.~Toulisse, and R.~Wentworth,
\emph{Higgs bundle, isomonodromic leaves and minimal surfaces},
arXiv:2512.07152v1, 2025.

\bibitem{CorletteSimpson08}
K.~Corlette and C.~Simpson,
\emph{On the classification of rank-two representations of quasiprojective fundamental groups},
Compos. Math. \textbf{144} (2008), no.~5, 1271--1331.

\bibitem{EngelTayou26}
P.~Engel and S.~Tayou,
\emph{On the non-abelian Hodge locus I},
Selecta Math. (N.S.) \textbf{32} (2026), Art.~80.

\bibitem{FuSheng}
Y.~Fu and M.~Sheng,
\emph{Nonabelian Kodaira--Spencer maps},
arXiv:2509.06050v2, 2026.

\bibitem{Goldman88}
W.~M.~Goldman,
\emph{Topological components of spaces of representations},
Invent. Math. \textbf{93} (1988), no.~3, 557--607.

\bibitem{Hitchin87}
N.~J.~Hitchin,
\emph{The self-duality equations on a Riemann surface},
Proc. London Math. Soc. (3) \textbf{55} (1987), no.~1, 59--126.

\bibitem{HSYZ}
T.~Hu, R.~Sun, J.~Yang, and K.~Zuo,
\emph{Isomonodromic deformations of Higgs bundles and characterization of the non-abelian Noether--Lefschetz locus},
arXiv:2606.18768v2, 2026.

\bibitem{HSZ}
T.~Hu, R.~Sun, and K.~Zuo,
\emph{Non-abelian Kodaira--Spencer map and non-existence of holomorphic isomonodromic deformation of Higgs bundles over Teichm\"uller spaces},
arXiv:2512.15478v3, 2026.

\bibitem{LPT}
F.~Loray, J.~V.~Pereira, and F.~Touzet,
\emph{Representations of quasi-projective groups, flat connections and transversely projective foliations},
J. \'Ec. polytech. Math. \textbf{3} (2016), 263--308.

\bibitem{Milnor58}
J.~Milnor,
\emph{On the existence of a connection with curvature zero},
Comment. Math. Helv. \textbf{32} (1958), 215--223.

\bibitem{Scott83}
P.~Scott,
\emph{The geometries of $3$-manifolds},
Bull. London Math. Soc. \textbf{15} (1983), no.~5, 401--487.

\bibitem{Simpson91}
C.~T.~Simpson,
\emph{The ubiquity of variations of Hodge structure},
in \emph{Complex Geometry and Lie Theory (Sundance, UT, 1989)},
Proc. Sympos. Pure Math., vol.~53,
Amer. Math. Soc., Providence, RI, 1991, pp.~329--348.

\bibitem{Simpson92}
C.~T.~Simpson,
\emph{Higgs bundles and local systems},
Publ. Math. Inst. Hautes \'Etudes Sci. \textbf{75} (1992), 5--95.

\bibitem{Simpson97}
C.~T.~Simpson,
\emph{The Hodge filtration on nonabelian cohomology},
in \emph{Algebraic Geometry---Santa Cruz 1995},
Proc. Sympos. Pure Math., vol.~62, part~2,
Amer. Math. Soc., Providence, RI, 1997, pp.~217--281.

\bibitem{Takeuchi71}
K.~Takeuchi,
\emph{Fuchsian groups contained in $SL_2(\Q)$},
J. Math. Soc. Japan \textbf{23} (1971), 82--94.

\bibitem{Troyanov91}
M.~Troyanov,
\emph{Prescribing curvature on compact surfaces with conical singularities},
Trans. Amer. Math. Soc. \textbf{324} (1991), no.~2, 793--821.

\bibitem{UY86}
K.~Uhlenbeck and S.-T.~Yau,
\emph{On the existence of Hermitian--Yang--Mills connections in stable vector bundles},
Comm. Pure Appl. Math. \textbf{39} (1986), S257--S293.

\bibitem{Wood71}
J.~W.~Wood,
\emph{Bundles with totally disconnected structure group},
Comment. Math. Helv. \textbf{46} (1971), 257--273.

\end{thebibliography}
\end{document}